\documentclass[a4paper]{scrartcl}
\usepackage[latin1]{inputenc}
\usepackage{geometry}
\usepackage{amsmath}
\usepackage{amsfonts}
\usepackage{amssymb}
\usepackage{amsthm}
\usepackage{multirow}
\usepackage{hyperref}
\usepackage[all,knot]{xy}
\usepackage{rotating}
\usepackage{tabularx,caption}
\usepackage{bbm}
\xyoption{arc}

  \theoremstyle{plain}
 \newtheorem{thm}{Theorem}[section]

 \theoremstyle{definition}
 \newtheorem{defn}[thm]{Definition}
 \theoremstyle{definition}
 \newtheorem{defnthm}[thm]{Definition and Theorem}
 \theoremstyle{definition}
 \newtheorem{prop}[thm]{Proposition} 
 \theoremstyle{definition}
 \theoremstyle{definition}
 
 \newtheorem{rem}[thm]{Remark}
 \theoremstyle{definition}
 \newtheorem{cor}[thm]{Corollary}
 \theoremstyle{definition}
 
 \theoremstyle{definition}
 
 \theoremstyle{definition}
 
 \theoremstyle{definition}
 
 \theoremstyle{definition}

\usepackage[usenames,dvipsnames]{color}

\author{Rachel Newton and Ana Ros Camacho}

\title{Strangely dual orbifold equivalence I}
\makeindex

\begin{document}
\maketitle

\abstractname{.~ In this brief note we prove orbifold equivalence between two potentials described by strangely dual exceptional unimodular singularities of type $K_{14}$ and $Q_{10}$ in two different ways. 
The matrix factorizations proving the orbifold equivalence give rise to equations whose solutions are permuted by Galois groups which differ for different expressions of the same singularity.}

\tableofcontents

\section{Introduction}

In this paper, we present two ways of proving an orbifold equivalence between two potentials describing two strangely dual unimodular exceptional singularities,
namely $Q_{10}$ and $K_{14}$. In addition, we observe that each matrix factorization proving this orbifold equivalence depends on a different Galois orbit. First, we will recall the notion of orbifold equivalence and motivate this research direction, leaving computations for Sections 3 and 4. We also include an appendix, written by the second author with Federico Zerbini, which discusses the Kreuzer--Skarke theorem and gives a way to count invertible potentials for any number of variables.

\subsection{Orbifold equivalence}

\quad We will work in the graded ring of polynomials over the complex numbers, $\mathbb{C} \left[ x_1,\ldots,x_n \right]$, with degrees $\vert x_i \vert \in \mathbb{Q}_{\geq 0}$ associated to each variable $x_i$.

\begin{defn}
A \textit{potential} is a polynomial $W \in \mathbb{C} \left[ x_1,\ldots,x_n \right]$ satisfying $$\mathrm{dim}_{\mathbb{C}} \left( \frac{\mathbb{C} \left[ x_1,\ldots,x_n \right]}{\langle \partial_1 W,\ldots,\partial_n W \rangle} \right) < \infty .$$ 
We say that a potential is \textit{homogeneous of degree $d\in\mathbb{Q}_{\geq 0}$} if in addition it satisfies  $$W \left( \lambda^{\vert x_1 \vert} x_1,\ldots,\lambda^{\vert x_n \vert} x_n \right)=\lambda^d W \left( x_1,\ldots,x_n \right)$$ for all $\lambda \in \mathbb{C}^{\times}$.
\end{defn}

From now on, the word potential will be used to mean `homogeneous potential of degree 2'. 

We will denote the set of all possible potentials with complex coefficients, and any number of variables, by $\mathcal{P}_\mathbb{C}$. To a potential $W \in \mathcal{P}_\mathbb{C}$ with $n$ variables, we can associate a number called the \textit{central charge}, which is defined as: $$c_W=\sum\limits_{i=1}^n \left( 1-\vert x_i \vert \right).$$

\begin{defn}
\begin{itemize}
\item A \textit{matrix factorization} of $W$ consists of a pair $\left( M, d^M \right)$ where
\begin{itemize}
\item $M$ is a $\mathbb{Z}_2$-graded free module over $\mathbb{C} \left[ x_1,\ldots,x_n \right]$;
\item $d^M: M \rightarrow M$ is a degree 1 $\mathbb{C} \left[ x_1,\ldots,x_n \right]$--linear endomorphism (the \textit{twisted differential}) such that:
\begin{equation}
d^M \circ d^M=W.\mathrm{id}_M.
\label{twistdiff}
\end{equation}
\end{itemize}
We may display the $\mathbb{Z}_2$-grading explicitly as $M=M_0 \oplus M_1$ and $d^M=\left( \begin{matrix} 0 & d_1^M \\ d_0^ M & 0 \end{matrix} \right)$.
If there is no risk of confusion, we will denote $\left( M,d^M \right)$ simply by $M$.

\item We call $M$ a \textit{graded matrix factorization} if, in addition, $M_0$ and $M_1$ are $\mathbb{Q}$-graded,  acting with $x_i$ is an endomorphism of degree $\vert x_i \vert$ with respect to the $\mathbb{Q}$-grading on $M$, and the twisted differential has degree 1 with respect to the $\mathbb{Q}$--grading on $M$. Note that these conditions imply that $W$ has degree 2 (as desired).
\end{itemize}
\end{defn}

 We will denote by $\mathrm{hmf}^{\mathrm{gr}} \left( W \right)$ the idempotent complete full subcategory of graded finite--rank matrix factorizations: its objects are homotopy equivalent to direct summands of finite--rank matrix factorizations. 
The morphisms are homogeneous even linear maps up to homotopy with respect to the twisted differential. This category is indeed monoidal and has duals and adjunctions which can be described in a very explicit way. This leads to the following result which gives precise formulas for the left and right quantum dimensions of a matrix factorization.

\begin{prop}{\cite{CM,rigidity}}
Let $V \left( x_1,\ldots,x_m \right)$ and $W \left( y_1,\ldots,y_n \right) \in \mathcal{P}_{\mathbb{C}}$ be two potentials and $M$ a matrix factorization of $W-V$. Then the left quantum dimension of $M$
is: $$\mathrm{qdim}_l \left( M \right)=\left( -1 \right)^{\binom{m+1}{2}} \mathrm{Res} \left[ \frac{\mathrm{str} \left( \partial_{x_1} d^M \ldots \partial_{x_m} d^M \partial_{y_1} d^M \ldots \partial_{y_n} d^M \right) d y_1 \ldots d y_n}{\partial_{y_1} W,\ldots,\partial_{y_n} W} \right]$$
and the right quantum dimension is:
$$\mathrm{qdim}_r \left( M \right)=\left( -1 \right)^{\binom{n+1}{2}} \mathrm{Res} \left[ \frac{\mathrm{str} \left( \partial_{x_1} d^M \ldots \partial_{x_m} d^M \partial_{y_1} d^M \ldots \partial_{y_n} d^M \right) d x_1 \ldots d x_m}{\partial_{x_1} V,\ldots,\partial_{x_m} V} \right].$$
\label{qdims}
\end{prop}

Quantum dimensions allow us to define the following equivalence relation:
\begin{defnthm}{\cite{CR,CRCR}}
Let $V$, $W$ and $M$ be as in the previous proposition. We say that $V$ and $W$ are \textit{orbifold equivalent} ($V \sim_{\mathrm{orb}} W$) if there exists a finite--rank matrix factorization of $V-W$ for which the left and the right quantum dimensions are non-zero. Orbifold equivalence is an equivalence relation in $\mathcal{P}_\mathbb{C}$.
\end{defnthm}

\begin{rem}{\cite[Proposition 6.4]{CR} (or \cite[Proposition 1.3]{CRCR}}) If two potentials $V$ and $W$ are orbifold equivalent, then their associated central charges are equal: $c_V=c_W$.
\label{charges}
\end{rem}

Let us give some comments on quantum dimensions and orbifold equivalences \cite{CRCR,CR}:
\begin{itemize}
\item \cite[Lemma 2.5]{CRCR} The quantum dimensions of graded matrix factorizations take values in $\mathbb{C}$. One can see this by counting degrees in the formulas given in Proposition \ref{qdims}.
\item The definitions of the quantum dimensions are also valid for ungraded matrix factorizations (in which case they will take values in $\mathbb{C} \left[ x_1,\ldots,x_n \right]$ instead of in $\mathbb{C}$). Furthermore, the quantum dimensions are independent of the $\mathbb{Q}$-grading on a graded 
matrix factorization.
\item So far, the difficulty of establishing an orbifold equivalence lies in constructing the explicit matrix factorization which proves it.

\end{itemize}

\subsection{Motivation: an interlude on Arnold's strange duality}

From now on, we fix the number of variables of our polynomial ring to be $n=3$.

\quad The aim of this work was to discover more orbifold equivalent potentials as in \cite{CRCR}. 
In that paper, orbifold equivalence between simple singularities was proven. These singularities have modality zero and fall into an ADE classification. A natural next step for finding new orbifold equivalences is to focus on potentials described by singularities of modality one. Thanks to the classification performed by Arnold in the late 60's, we know that such singularities fall into 3 families of parabolic singularities, a three-suffix series of hyperbolic singularities, and 14 families of exceptional singularities. For more details on this classification, we refer to \cite{Ar,AGV}.

\quad A singularity can be described with a regular weight system \cite{Sai}, that is, a quadruple of positive integers $\left( a_1, a_2,a_3; h \right)$ with:
\begin{itemize}
\item[--] $a_1, a_2, a_3 < h$,
\item[--] $\mathrm{gcd} \left( a_1,a_2,a_3 \right)=1$, and
\item[--] There exists a polynomial $W \in \mathbb{C} \left[ x_1,x_2,x_3 \right]$  that has an isolated singularity at the origin (with the degrees of the variables $x_i$ being $\vert x_i \vert=\frac{2 a_i}{h}$, $i \in \lbrace 1,2,3 \rbrace$) which is invariant under the Euler field $E$, that is, $$E.W=\left( \frac{a_1}{h} x_1 \frac{\partial}{\partial x_1}+\frac{a_2}{h} x_2 \frac{\partial}{\partial x_2}+\frac{a_3}{h} x_3 \frac{\partial}{\partial x_3} \right) W=W.$$ 

In other words, the polynomial associated to a regular weight system must be a potential invariant under the Euler field.
\end{itemize}

With the assignment of degrees made, this is the same as requesting homogeneity of degree 2 for the potentials \footnote{This argument goes as follows: a potential in three variables can only have seven possible shapes, which are specified in a graphical way in Table \ref{3vars} in the Appendix \ref{appendix}, or in \cite[Chapter 13]{AGV}. Imposing invariance under the Euler field boils down to some conditions on the powers of the monomials in the potential. With the assignment of degrees made, one can easily see that these conditions are exactly the same as those we should impose if we want homogeneity of degree 2.}.
The integer $h$ is called the \textit{Coxeter number}.

From now on, we write $x_1=x$, $x_2=y$ and $x_3=z$. Some examples of regular weight systems, those corresponding to each of the 14 unimodular exceptional singularities are shown in Table \ref{Excepts1}. 
The associated potentials are also described. For most of the exceptional unimodular singularities, there is only one way to write the associated potential, whereas there are two expressions for each of $Q_{12}$, $Z_{13}$, $W_{12}$, $W_{13}$ and $K_{14}$. Exceptionally, there are 4 potentials 
which can describe the singularity $U_{12}$. In order to find these potentials, combine invariance under the Euler field (or homogeneity of order 2) with the Kreuzer--Skarke theorem \cite{KS} to see that any variable $x_i$ shows up in a potential only as a power of itself, $x_i^a$ (for some $a>2$) or as $x_i^a x_j$ (with $i \neq j$). \footnote{A complete statement of this theorem, as well as a discussion of it, is presented in the Appendix \ref{appendix}.}.

Let us illustrate this with an example: take $K_{14}$. The degrees assigned to the variables are: $\vert x \vert=\frac{6}{24}=\frac{1}{4}$, $\vert y \vert=\frac{16}{24}=\frac{2}{3}$ and $\vert z \vert=\frac{24}{24}=1$. Imposing homogeneity of degree 2, we need to find monomials of the shape $x^{k_1} y^{k_2} z^{k_3}$ where $k_i \in \mathbb{Z}_+$, $i \in \lbrace 1,2,3 \rbrace$ must satisfy $\frac{2}{3} k_1+\frac{1}{4} k_2+k_3=2$. The only solutions are four tuples: $\left( 8,0,0 \right)$, $\left( 4,0,1 \right)$, $\left( 0,3,0 \right)$, $\left( 0,0,2 \right)$, i.e. the monomials $x^4 z$, $x^8$, $y^3$ and $z^2$. Combining them and taking into account the Kreuzer-Skarke theorem, we get the two potentials appearing in Table \ref{Excepts1}.

\begin{table}
\begin{center}
\begin{tabular}{c|c|c|c}
Type & Potential (1) & Potential (2) & $\left( a_1, a_2,a_3; h \right)$ \\ \hline
$Q_{10}$ & $x^4+y^3+x z^2$ & -- & $\left( 9,8,6;24 \right)$ \\
$Q_{11}$ & $x^3 y+y^3+x z^2$ & -- & $\left( 7,6,4;18 \right)$ \\
$Q_{12}$ & $x^3 z+y^3+x z^2$ & $x^5+y^3+x z^2$ & $\left( 6,5,3;15 \right)$ \\
$S_{11}$ & $x^4+y^2 z+x z^2$ & -- & $\left( 5,4,6;16 \right)$ \\
$S_{12}$ & $x^3 y+y^2 z+x z^2$ & -- & $\left( 4,3,5;13 \right)$ \\
$U_{12}$ & $x^4+y^3+z^3$ & $x^4+y^3+z^2 y$& $\left( 4,4,3;12 \right)$ \\
$Z_{11}$ & $x^5+x y^3+z^2$ & -- & $\left( 8,6,15;30 \right)$ \\
$Z_{12}$ & $y x^4+x y^3+z^2$ & -- & $\left( 6,4,11;22 \right)$ \\
$Z_{13}$ & $x^3 z+x y^3+z^2$ & $x^6+y^3 x+z^2$ & $\left( 5,3,9;18 \right)$ \\
$W_{12}$ & $x^5+y^2 z+z^2$ & $x^5+y^4+z^2$ & $\left( 5,4,10;20 \right)$ \\
$W_{13}$ & $y x^4+y^2 z+z^2$ & $x^4 y+y^4+z^2$ & $\left( 4,3,8;16 \right)$ \\
$K_{12}$ & $x^7+y^3+z^2$ & -- & $\left( 14,6,21;42 \right)$ \\
$K_{13}$ & $y^3+y x^5+z^2$ & -- & $\left( 10,4,15;30 \right)$ \\
$K_{14}$ & $x^4 z+y^3+z^2$ & $x^8+y^3+z^2$ & $\left( 8,3,12;24 \right)$ \\ \hline
\end{tabular}
\caption{Unimodular singularities of exceptional type (note that $U_{12}$ can also be described in two additional different ways: $x^4+y^2 z+z^3$ and $x^4+y^2 z+z^2 y$).}
\label{Excepts1}
\end{center} 
\end{table}

\quad As discovered by Kobayashi \cite{Kob}, there is some duality between these weight systems -- which corresponds to what is known as \textit{Arnold's strange duality}\footnote{This duality roughly states that, given two singularities, the Dolgachev numbers associated to the first singularity are the same as the Gabrielov numbers of the second one (and vice versa). We refer to the bibliography for further details, e.g. \cite{Ar,Dol,Eb}.}. Four pairs of these exceptional singularities share the same Coxeter number: $Q_{10}$ and $K_{14}$ ($h=24$), $Q_{11}$ and $Z_{13}$ ($h=18$), $S_{11}$ and $W_{13}$ ($h=16$) and $Z_{11}$ and $K_{13}$ ($h=30$).

\quad In addition, one notices the following phenomenon. For potentials described by strange dual pairs, the associated central charges have a close relationship with the Coxeter number $h$ \cite{Ma2}, $$c_W=\frac{h+2}{h}$$ which implies that the potentials related to strange dual singularities have the same central charge. As mentioned in Remark \ref{charges}, equality of central charges is one consequence of orbifold equivalence between two potentials. Hence, it makes sense to conjecture from the mathematics point of view that strangely dual exceptional unimodular singularities are orbifold equivalent.

\quad Another
consequence of orbifold equivalence between strangely dual exceptional unimodular singularities would be that the Ginzburg algebras \cite{Ginz} for these singularities with Dynkin diagrams \cite{Gab} sharing the same Coxeter number are orbifold equivalent in the bicategory whose objects are smooth dg algebras with finite dimensional cohomology and whose morphism categories are the respective perfect derived categories. We refer to the recent paper \cite{CQ} for a complete exposition and details of this statement.

\quad Furthermore, from the physics point of view, we have known for some time that for each of these exceptional singularities there is a uniform construction of a $K3$ surface obtained by compactifying the singularity \cite{Sai,Pin}. Landau-Ginzburg models with potentials described by strangely dual singularities correspond to the same $K3$ surface \cite{Ma1,Ma2}. This can also be regarded as well 
as a prediction of orbifold equivalence between these singularities. In addition, it would be interesting to see the implications of orbifold equivalence for $N=2$ superconformal four--dimensional gauge theories \cite{CDZ}.

\quad A further motivation for this work (if not the primary for the second author) 
is given by the so-called Landau-Ginzburg/conformal field theory correspondence \cite{howewest,lerchevafawarner,vafawarner,ARC}, which predicts a certain relation between categories of matrix factorizations of the potential of the Landau-Ginzburg model and categories of representations of the vertex operator algebra associated to some conformal field theory. An immediate consequence of orbifold equivalence between two potentials is the following result:

\begin{prop}{\cite{CR}}
Let $V$, $W \in \mathcal{P}_\mathbb{C}$ be two potentials which are orbifold equivalent and let $M \in \mathrm{hmf}^{\mathrm{gr}} \left( W-V \right)$ have non-zero quantum dimensions. Then,
\begin{equation}
\mathrm{hmf}^\mathrm{{gr}} \left( W \right) \simeq \mathrm{mod} \left( X^ \dagger \otimes X \right)
\label{equivalence}
\nonumber
\end{equation}
where by $X^\dagger$ we mean the right adjoint of $X$ and $\mathrm{mod} \left( X^ \dagger \otimes X \right)$ is the category of modules over $X^ \dagger \otimes X$.
\end{prop} 
$X^ \dagger \otimes X$ is
a separable symmetric Frobenius algebra \cite{CR} (see e.g. \cite{BCP} for a review on Frobenius algebras). These algebras are related to full CFTs \cite{tft1}. Hence, proving orbifold equivalences is a way to match together both sides of the Landau-Ginzburg/conformal field theory correspondence, providing a better understanding of a mathematical conjecture for it. Due to the need for computational software improvements, we postpone the analysis of the results of this paper from the point of view of the Landau-Ginzburg/conformal field theory correspondence to later works \cite{NRC}.

Proving more orbifold equivalences requires at this point some strong computational tool which for the moment we lack \footnote{Upon the writing of this manuscript, the second author became aware of a project by Andreas Recknagel et al. to create a computer algorithm to prove orbifold equivalences. We do not know any further details about this project, but it seems that this algorithm was able to reproduce the orbifold equivalences of \cite{CRCR} and the one in this paper as well - apparently via a different method but nonetheless pretty simultaneously.}. For this reason we focus on a first example -- that of $K_{14}-Q_{10}$ -- and analyze it in detail.

This paper is organized as follows. In Section 2, we explain orbifold equivalence as well as some basics on matrix factorizations. In Section 3, we describe the method followed to find the matrix factorizations of $K_{14}-Q_{10}$ 
which prove 
orbifold equivalence in two different fashions. In Section 4, we describe the Galois orbits on which the matrix factorizations obtained in Section 3 depend. We wrap up with some conclusions and an appendix by the second author and Federico Zerbini on the Kreuzer--Skarke theorem.

\subsection*{Acknowledgments}

The authors are grateful to the Max-Planck-Institut f{\"u}r Mathematik in Bonn (Germany) for providing the best possible working conditions for this collaboration.
In addition, ARC wishes to thank
Ingo Runkel,
Nils Carqueville, 
Atsushi Takahashi and Lev Borisov for very useful discussions and feedback on this paper. Especial thanks are for all of the subsets of the set $\lbrace$Sonny John Moore, Thomas Wesley Pentz$\rbrace$ and the pair (Dylan Mamid, Zach Rapp-Rovan), which provided an awesome soundtrack to this work.

\section{$\mathbf{Q_{10} \sim_{\mathrm{orb}} K_{14}}$ in two fashions}

Our method to find matrix factorizations of finite rank consists of a variation of the perturbation method used in \cite{CRCR}. The starting point is the paper \cite{KST}, 
where we find the full strongly exceptional collection of objects of the category of matrix factorizations of each potential described by unimodular singularities. Our recipe proceeds as follows:

\begin{enumerate}
\item Consider the difference between two potentials. Set to zero one of the variables (typically the one with the smallest degree associated). Factorize the resulting potential.
\item Pick one of the exceptional objects from the \cite{KST} collection for the potential which doesn't contain the variable set to zero in the previous step. The entries of these matrices are factorizations of each of the monomials of the corresponding potential. We change these factorizations in order to obtain entries in the matrix similar to the factors in the factorization of Step 1, being careful to ensure that the result is still a matrix factorization.
\item Perturb {\`a} la \cite{CRCR} all possible entries of the matrix factorization (not necessarily only with respect to the variable set to zero), except for the zero entries.
\item Impose Equation \ref{twistdiff} and reduce the system of equations obtained from the perturbation constants as much as possible. We obtain a matrix factorization depending on a small number of parameters satisfying some equations.
\end{enumerate}

In an attempt to elucidate this recipe, we will explain in detail how to prove $Q_{10} \sim_{\mathrm{orb}} K_{14}$ in two ways.

\subsection{$\mathbf{Q_{10} \sim_{\mathrm{orb}} K_{14}}$, version 1}\label{version1}

\begin{enumerate}
\item Consider the potentials:
$$Q_{10}=x^4+y^3+x z^2$$
$$K_{14}=u^4 w+v^3+w^2$$
whose variables have the following associated degrees:
\begin{equation}
\vert x \vert=\frac{6}{12} \quad \quad \vert y \vert=\frac{8}{12} \quad \quad \vert x \vert=\frac{9}{12} \quad \quad \vert u \vert=\frac{3}{12} \quad \quad \vert v \vert=\frac{8}{12} \quad \quad \vert w \vert=\frac{12}{12}. 
\nonumber
\end{equation}
It is easy to check that both potentials have a central charge of $c_{Q_{10}}=\frac{13}{12}=c_{K_{14}}.$
The variable with the smallest degree is $u$ and we will perturb with respect to it. Set $u$ equal to zero; the resulting potential is then:
$$\overline{Q_{10}-K_{14}}=x^4+y^3+x z^2-v^3-w^2$$
We can factorize this potential as:
\begin{equation}
\overline{Q_{10}-K_{14}}=\left( x^2+w \right) \left( x^2-w \right)+\left(y-v \right) \left( y^2+y v +v^2\right) + \left( x z \right) \left( z \right).
\label{factorization}
\end{equation}
\item First, we will start from the indecomposables of $Q_{10}$. The matrix factorization associated to the vertex $V_0$ of the Auslander-Reiten quiver associated to this singularity is given by (\cite{KST}):
\begin{equation}
d_0=\left( \begin{matrix}
x z & y^2 & x^3 & 0 \\ y & -z & 0 & x^3 \\ x & 0 & -z & -y^2 \\ 0 & x & -y & x z 
\end{matrix} \right) \quad \quad d_1=\left( \begin{matrix}
z & y^2 & x^3 & 0 \\ y &- x z & 0 & x^3 \\ x & 0 & -x z & -y^2 \\ 0 & x & -y & z
\end{matrix} \right)
\nonumber
\end{equation}
Note that the determinant of $d_1$ is precisely $Q_{10}^2$. Then, similarly to the procedure followed to prove the orbifold equivalence $A_{29} \sim_{\mathrm{orb}} E_8$ in \cite{CRCR}, we make the ansatz that it is possible to recover $d_0$ as $Q_{10} d_1^{-1}$. Hence we will only need to work with $d_1$.
Modify $d_1$ as follows:
$$\widetilde{d_1}=\left( \begin{matrix}
z & y^2 & x^2 & 0 \\ y &- x z & 0 & x^2 \\ x^2 & 0 & -x z & -y^2 \\ 0 & x^2 & -y & z
\end{matrix} \right)$$
The determinant of this matrix is still equal to $Q_{10}^2$. Then, using the factorization in Eq. \ref{factorization}, we can construct a similar $d_1$ whose determinant is precisely $\left( \overline{Q_{10}-K_{14}} \right)^2$:
\begin{equation}
\widetilde{\widetilde{d_1}}=\left( \begin{matrix}
z & v^2+v y+y^2 & x^2+w & 0 \\ y-v &- x z & 0 & x^2+w \\ x^2-w & 0 & -x z & -\left(v^2+y v+ y^2 \right) \\ 0 & x^2-w & -y+v & z
\end{matrix} \right)
\label{pertfact}
\nonumber
\end{equation}
which has a degree distribution (in units of 1/12) specified in Table \ref{degreestable}.
\begin{table}
\begin{center}
\begin{tabular}{c|c|c|c}
9 & 16 & 12 &0 \\ \hline
8 & 15 & 0 & 12 \\ \hline
12 & 0 & 15 & 16 \\ \hline
0 & 12 & 8 & 9
\end{tabular}
\caption{Degree distribution of the entries of $\tilde{\tilde{d_1}}$}
\label{degreestable}
\end{center}
\end{table}

From this matrix, construct $\widetilde{\widetilde{d_0}}$:
\begin{equation}
\widetilde{\widetilde{d_0}}=\left( \begin{matrix}
-x z & - \left( v^2+v y+y^2 \right) & -\left( x^2+w \right) & 0 \\ v-y & z & 0 & -\left( x^2+w \right) \\ -\left( x^2-w \right) & 0 & z & v^2+y v+ y^2 \\ 0 & -\left(x^2-w\right) & -v+y & -x z
\end{matrix} \right)
\label{pertfact0}
\nonumber
\end{equation}
which has a degree distribution (in units of 1/12) specified in Table \ref{degreestable0}.
\begin{table}
\begin{center}
\begin{tabular}{c|c|c|c}
15 & 16 & 12 &0 \\ \hline
8 & 9 & 0 & 12 \\ \hline
12 & 0 & 9 & 16 \\ \hline
0 & 12 & 8 & 15
\end{tabular}
\caption{Degree distribution of the entries of $\tilde{\tilde{d_0}}$}
\label{degreestable0}
\end{center}
\end{table}

Now form the whole matrix factorization (which we will denote by $d_X$). Indeed, we see that $d_X \circ d_X=\overline{Q_{10}-K_{14}}$.
\item Perturb all possible entries with terms (at least) linear in $u$. Note that, in contrast to \cite{CRCR}, the zero entries are not perturbed. Those which can be perturbed in this way are those of degree:
\begin{itemize}
\item 9: $u^3$, $u x$.
\item 12: $u z$, $u^4$, $x u^2$.
\item 15: $u x^2$, $u^2 z$, $u w$, $u^5$.
\end{itemize}
\end{enumerate}
Implement the perturbation in $d_X=\widetilde{\widetilde{d_0}} \oplus \widetilde{\widetilde{d_1}}=\left( x_{ij} \right)$ ($i,j=1,\ldots,8$); the entries of this matrix will be 
\begin{equation}
\begin{split}
x_{15} &= z+p_{111} u^3+p_{112} u x \\
x_{16} &= v^2+v y+y^2\\
x_{17} &= x^2+w+ p_{131} u z+p_{132} u^4+p_{133} x u^2\\
x_{25} &= y-v \\
x_{26} &= -xz + p_{221} u x^2 + p_{222} u^2 z + p_{223} u w + p_{224} u^5\\
x_{28} &= w + x^2 + p_{241} u z + p_{242} u^4 + p_{243} x u^2\\
x_{35} &= -w + x^2 + p_{311} u z + p_{312} u^4 + p_{313} x u^2\\
x_{37} &= -x z + p_{331} u x^2 + p_{332} u^2 z + p_{333} u w + p_{334} u^5\\
x_{38} &= -v^2 -v y - y^2 \\
x_{46} &= -w + x^2 + p_{421} u z + p_{422} u^4 + p_{423} x u^2\\
x_{47} &= v - y\\
x_{48} &=  z + p_{441} u^3 + p_{442} u x\\
\end{split}
\nonumber
\end{equation}
for $d_1$, and similarly for $d_0$,
with the rest of entries of the matrix zeros and where $p_{lmn} \in \mathbb{C}$ ($l=1,\ldots,8$; $m,n=1,\ldots,4$). Imposing Equation \ref{twistdiff} and linear conditions on the $p_{ijk}$'s, we finally recover a diagonal matrix where in order to recover the original potential $Q_{10}-K_{14}$ we need to solve a system of 11 equations with 12 variables, 
which can indeed be further reduced. Changing $p_{112} \rightsquigarrow a$, $p_{131} \rightsquigarrow b$ and $p_{221} \rightsquigarrow c$, we are left with only two equations and three variables,

\begin{equation}
\begin{split}
& -\frac{1}{64} \left( -4 + 3 a^4 + 8 a^3 b + 8 a^2 b^2 - 4 a^3 c - 8 a^2 b c \right) \\ & \quad \cdot \left( 4 + 3 a^4 + 8 a^3 b + 8 a^2 b^2 - 4 a^3 c - 8 a^2 b c \right)=0\\ 
& -\frac{1}{8} a^2 \left( a^4 - 8 a^2 b^2 - 16 a b^3 - 8 b^4 + 8 a^2 b c + 24 a b^2 c + 16 b^3 c - 2 a^2 c^2 - 8 a b c^2 - 8 b^2 c^2 \right)=0
\end{split}
\label{Galoisorbit}
\end{equation}

For the sake of simplification, introduce the following notation:
$$\kappa_1:=\left( \frac{a^3}{2} + a^2 b + a b^2 - \frac{a^2 c}{2} - a b c \right)$$
$$\kappa_2:=1 + \frac{3 a^4}{4} + 3 a^3 b + 4 a^2 b^2 + 2 a b^3 - a^3 c - 
      3 a^2 b c - 2 a b^2 c$$
      
The entries of $d_X$ finally look like:
\begin{equation}
\begin{split}
x_{15} &= \kappa_1 u^3 + a u x + z,
\\
x_{16} &= v^2+v y+y^2,\\
x_{17} &= \frac{1}{2} \kappa_2 u^4 + w - \frac{1}{2} a \left(-a - 2 b \right) u^2 x + x^2 + 
   b u z,\\
x_{25} &= y-v, \\
x_{26} &= \left(-b - b^2 \kappa_1+ 
      \frac{1}{2} \left(c-a \right) \kappa_2 \right) u^5 + \left( -a - 2 b + c \right) u w \\ &+ c u x^2 + 
   b \left(-a - b + c \right) u^2 z - x z,\\
x_{35} &= \left( -1 + \left(-a - 2 b + c \right) \kappa_1 +\frac{\kappa_2}{2}  \right) u^4 \\ &- w + 
   \frac{1}{2} a \left(-a - 2 b + 2 c \right) u^2 x + x^2 + \left(-a - b + c \right) u z, 
\end{split}
\nonumber
\end{equation}
with 
\begin{equation}
\begin{split}
x_{15} &=x_{48}=x_{62}=x_{73} \\
x_{16} &=-x_{38}=-x_{52}=x_{74} \\
x_{17} &=x_{28}=-x_{53}=-x_{64} \\
x_{25} &=-x_{47}=-x_{61}=x_{83} \\
x_{26} &=x_{37}=x_{84}=x_{51} \\
x_{35} &=x_{46}=-x_{71}=-x_{82} 
\end{split}
\nonumber
\end{equation}
and with all other entries of the matrix zero.

The quantum dimensions of our matrix factorization are

\begin{equation}
\begin{split}
\mathrm{qdim}_l \left( d_X \right) &= \frac{1}{2} a^2 \left( a + 2 b - c \right) \\
\mathrm{qdim}_r \left( d_X \right) &=-2 \left( a - c \right)
\end{split}
\nonumber
\end{equation}

which are not zero for any values of $a$, $b$, $c$ satisfying Eqs. \ref{Galoisorbit}.

\subsection{$\mathbf{Q_{10} \sim_{\mathrm{orb}} K_{14}}$, version 2}\label{version2}

\begin{enumerate}
\item This time we consider the potentials:
$$Q_{10}=x^4+y^3+x z^2$$
$$K_{14}=u^3 + v^8 + w^2$$
that is, the same $Q_{10}$ but a different $K_{14}$. The variables of the potential $Q_{10}$ have the same associated degree, while $u$ and $v$ of $K_{14}$ switch theirs. This time, we will perturb with respect to $w$ (the variable with the biggest degree). Set it equal to zero, and the resulting potential is:
$$\overline{Q_{10}-K_{14}}=x^4+y^3+x z^2 -u^3 - v^8$$
which has again a factorization similar to that of Eq. \ref{factorization}:
$$\overline{Q_{10}-K_{14}}=\left( x^2 + v^4 \right) \left( x^2 -v^4 \right)+\left( y-u \right) \left( y^2+y u+u^2 \right)+\left( x z \right) \left( z \right)$$

\item Proceeding analogously to \ref{version1}, we get:
$$\tilde{\tilde{d_1}} = \left( \begin{matrix} z & u^2 + u y + y^2 & v^4 + x^2 & 0 \\ -u + y & -x z & 0 & v^4 + x^2 \\ -v^4 + x^2 & 0 & -x z & -u^2 - u y - y^2 \\ 0 & -v^4 + x^2 & u - y & z \end{matrix} \right)$$
whose determinant is precisely $\overline{Q_{10}-K_{14}}^2$. The degrees are distributed in the matrix in the same way as in Table \ref{degreestable}. Again, $\tilde{\tilde{d_0}}$ is given by $\overline{Q_{10}-K_{14}} \tilde{\tilde{d_1}}^{-1}$.

\item In this case, we will allow all possible perturbations -- not only those linear in $w$. The perturbations associated to each degree are then:
\begin{itemize}
\item 9: $v^3$, $v x$.
\item 12: $v z$, $v^2 x$, $w$.
\item 15: $v w$, $v^5$, $v^2 z$, $v^3 x$, $v x^2$.
\end{itemize}
We proceed as in the previous example. We obtain a matrix factorization with entries:
\begin{equation}
\begin{split}
x_{15} &=b v^3 + c v x + z, \\
x_{16} &=u^2 + u y + y^2, \\
x_{17} &=v^4 + a w + \frac{1}{2} \left( c^2 + 2 c d \right) v^2 x + x^2 + d v z,\\
x_{25} &=-u + y,\\ 
x_{26} &= -\frac{2 a v w}{b} + \left( b + \frac{2 c}{b^2} - \frac{2 c d}{b} + c^2 d + 2 c d^2 - \frac{c^2 + 2 c d}{b} \right) v^3 x  \\ &+ \left(-\frac{2}{b} + c + 2 d \right) v x^2 - \frac{2 v^2 z}{b^2} - x z,\\ 
x_{35} &=-v^4 - a w + \left( c \left( -\frac{2}{b} + c + 2 d \right) + \frac{1}{2} \left( -c^2 - 2 c d \right) \right) v^2 x  \\ &+ 
   x^2 + \left(-\frac{2}{b} + d \right) v z,
\end{split}
\nonumber
\end{equation}
and 
\begin{equation}
\begin{split}
x_{15} &=x_{48}=x_{62}=x_{73} \\
x_{16} &=-x_{38}=-x_{52}=x_{74} \\
x_{17} &=x_{28}=x_{53}=-x_{64}=x_{82} \\
x_{25} &=-x_{47}=-x_{61}=x_{83} \\
x_{26} &=x_{37}=x_{26}=x_{51}=x_{84} \\
x_{35} &=x_{46}=-x_{71}
\end{split}
\nonumber
\end{equation}
with the rest of the entries of the matrix factorization being zero.

$a$, $b$, $c$ and $d$ must satisfy:
\begin{equation}
\begin{split}
a^2 &=1 \\
b^2 + \frac{4 c}{b} - c^2 - 4 c d + b c^2 d + 2 b c d^2 &=0 \\ 
-2 + 2 b c + \frac{2 c^2}{b^2} - \frac{c^4}{4} + 2 b d - \frac{2 c^2 d}{b} + c^2 d^2 &=0 \\
\frac{-2}{b^2} + \frac{2 d}{b} - d^2&=0
\end{split}
\label{Galoisorbit2}
\end{equation}
\end{enumerate}

The quantum dimensions of this matrix factorization are:
\begin{equation}
\begin{split}
\mathrm{qdim}_l \left( d_X \right) &= \frac{ 24 a \left(-1 + b c + b d \right)}{b} \\
\mathrm{qdim}_r \left( d_X \right) &=\frac{6 a}{b^2} \left(-3 b^3 - 12 c + 7 b c^2 + 3 b^4 d + 24 b c d - 6 b^2 c^2 d - 18 b^2 c d^2 + 3 b^3 c^2 d^2 + 6 b^3 c d^3 \right)
\end{split}
\nonumber
\end{equation}

which are not zero for any values of $a$, $b$, $c$, $d$ which satisfy Eqs. \ref{Galoisorbit2}.

\section{Galois theory}

In this section, we analyze in detail the solutions of Eqs. \ref{Galoisorbit} and \ref{Galoisorbit2}. These solutions lie in Galois orbits, which are described in the following two propositions.

\begin{prop}
The solutions of Eqs. \ref{Galoisorbit} are permuted by a Galois group isomorphic to $D_8 \times C_2$. Moreover, the solutions comprise three distinct orbits for the Galois action. 

\end{prop}

\begin{proof}
Define
\begin{eqnarray*}
f_1&=&4 + 3 a^4 + 8 a^3 b + 8 a^2 b^2 - 4 a^3 c - 8 a^2 b c \\
f_2&=&f_1-8=-4 + 3 a^4 + 8 a^3 b + 8 a^2 b^2 - 4 a^3 c - 8 a^2 b c\\
g&=&a^4 - 8 a^2 b^2 - 16 a b^3 - 8 b^4 + 8 a^2 b c + 24 a b^2 c + 16 b^3 c - 2 a^2 c^2 - 8 a b c^2 - 8 b^2 c^2.\\
\end{eqnarray*} 
Eqns. \ref{Galoisorbit} reduce to $f_1f_2=g=0$. Thus, the solutions to Eqns. \ref{Galoisorbit} come in two disjoint families. Family 1 consists of solutions to $f_1=g=0$, and Family 2 consists of solutions to $f_2=g=0$.

Solving the equations shows that the solutions in Family 1 have $a=i^k\sqrt[4]{-12\pm 8\sqrt{2}}$ for some $k\in\mathbb{Z}/4\mathbb{Z}$, and all eight possibilities for $a$ occur. In other words, $a$ is a root of $x^8+24x^4+16$, which is irreducible over $\mathbb{Q}$. 

Solutions in Family 2 have $a=i^k\sqrt[4]{12\pm 8\sqrt{2}}=i^k\sqrt{2\pm\sqrt{2}}$, for some $k\in\mathbb{Z}/4\mathbb{Z}$, and all eight possibilities for $a$ occur.  in other words $a$ is a root of $$x^8-24x^4+16=(x^4-4x^2-4)(x^4+4x^2-4)=0.$$
The family of solutions with $a$ a root of the irreducible polynomial $x^4-4x^2-4$ will be called Family 2A. The solutions with $a$ a root of the irreducible polynomial $x^4+4x^2-4$ will be called Family 2B.

Every solution $(a,b,c)$ to Eqs. \ref{Galoisorbit} has $a$ defined over $L=\mathbb{Q}(\sqrt[4]{-3+2\sqrt{2}},\sqrt{1+\sqrt{2}})$ and, moreover, the values of $a$ for all solutions of Eqs. \ref{Galoisorbit} generate $L/\mathbb{Q}$. The field $L$ is a degree $16$ Galois extension of $\mathbb{Q}$ whose Galois group is isomorphic to $D_8\times C_2$ and has generators $\rho, \sigma,\tau$ with the following actions on $m=\sqrt[4]{-3+2\sqrt{2}}$ and $n=\sqrt{1+\sqrt{2}}$:
\begin{eqnarray*}
\rho:& m\mapsto im^{-1}, & n\mapsto in^{-1} \\
\sigma:&  m\mapsto m^{-1}, & n\mapsto in^{-1} \\
\tau:& m\mapsto m, & n\mapsto -n.
\end{eqnarray*}
Note that $i=(m^2+m^{-2})/2$, so $\rho$ has order $4$, whereas $\sigma$ and $\tau$ have order $2$. 


The $a$-values of solutions in Family 1 generate $\mathbb{Q}(m)/\mathbb{Q}$, the fixed field of $\tau$. The \mbox{$a$-values} of solutions in Family 2 generate $\mathbb{Q}(i,n)/\mathbb{Q}$, the fixed field of $\tau\rho^2$. Both $\mathbb{Q}(m)/\mathbb{Q}$ and  $\mathbb{Q}(i,n)/\mathbb{Q}$ are Galois extensions with Galois groups isomorphic to $D_8$. 

The solutions (a,b,c) in Family 1 satisfy the equations $a^8+24a^4+16=0$ and \mbox{$16(a+2b)c=32ab+32b^2-12a^2-a^6$.} They make up one Galois orbit.

The solutions (a,b,c) in Family 2A satisfy $a^4-4a^2-4=0$ and $2(a+2b)c=(a+2b)^2+2$. They make up one Galois orbit. The solutions (a,b,c) in Family 2B satisfy $a^4+4a^2-4=0$ and $2(a+2b)c=(a+2b)^2-2$. They make up one Galois orbit. 
\end{proof}



\begin{prop}
The solutions of Eqs. \ref{Galoisorbit2} are permuted by a Galois group isomorphic to $V_4=C_2\times C_2$. The solutions comprise eight orbits for the Galois action, with each orbit having $4$ elements.
\end{prop}

\begin{proof}
The solutions of Eqs. \ref{Galoisorbit2} consist of two families: solutions in Family(+1) have $a=1$, whereas solutions in Family(-1) have $a=-1$. We define a new variable $t$ by $t=bd$. The last equation in Eqs. \ref{Galoisorbit2} becomes \begin{equation}
\label{eq:t}
t^2-2t+2=0
\end{equation} and hence $t=1\pm i$.
Substituting \eqref{eq:t} into the second and third equations in Eqs. \ref{Galoisorbit2} and simplifying gives the following equivalent system of equations.

\begin{equation}
\label{neweqs}
\begin{split}
a^2 &=1 \\
\left(\frac{b}{c}\right)^2  &= 1-t \\ 
c^4-8\left(\frac{b}{c}\right)c^2+8\left(\frac{b}{c}\right)^2 &=0 \\
t^2-2t+2&=0.
\end{split}
\end{equation}
Hence, the solutions only depend on $a,b$ and $c$, and $b/c$ is a primitive $8$th root of unity. The solutions for $c$ are the roots of 
$f(x)=x^{16}+2^7.17x^8+2^{12}$, which decomposes into four quartic polynomials over $\mathbb{Q}$, and splits completely into linear factors over $\mathbb{Q}(\zeta_8)$. Therefore, all values of $c$ are defined over $\mathbb{Q}(\zeta_8)$, which has Galois group $V_4$.
For each value of $c$, there is a unique primitive $8$th root of unity $\beta$ such that $c^4-8\beta c^2+8\beta^2 =0$. In other words, each value of $c$ determines a value of $b/c$, and hence also a value of $t$.

Each family of solutions, Family(+1) and Family(-1), breaks down into four Galois orbits, one for each quartic factor in the decomposition of $f$ over $\mathbb{Q}$. So, in total we have eight Galois orbits, each with four elements corresponding to the four roots of a quartic factor of $f$.
\end{proof}


\begin{rem}
Note the marked differences between the solutions of Eqs.  \ref{Galoisorbit}  and those of Eqs.  \ref{Galoisorbit2}. In particular, there are infinitely many solutions to Eqs.  \ref{Galoisorbit}, whereas Eqs.  \ref{Galoisorbit2} admit precisely 32 solutions.
\end{rem}

The elements in the Galois group interfere with our matrix factorizations in the following way. Let $W \in \mathbb{Q} \left[ x,y,z \right]$ be a potential and let $M$ be a finite--rank matrix factorization of~$W$ given by $\left( \mathbb{C} \left[ x,y,z \right]^{\oplus 2 r}, d^M \right)$ ($r \in \mathbb{N}$). Let $\sigma $ be an element of the Galois group and denote by $\sigma \left( d_M \right)$ the twisted differential obtained by applying $\sigma$ to each entry. Since $\sigma$ leaves the potential invariant, i.e. $\sigma \left( W \right) = W$, $\sigma \left( d_M \right)$ is still a factorization of $W$, $\sigma \left( M \right) = \left( \mathbb{C} \left[ x,y,z \right]^{\oplus 2r}, \sigma \left( d_M \right) \right)$. Therefore, we obtain not only one matrix factorization proving orbifold equivalence between $Q_{10}$ and $K_{14}$, but infinitely many for Eqs. \ref{Galoisorbit} and 32 for Eqs. \ref{Galoisorbit2} -- one for each solution.

\begin{rem}
Note that the two Galois groups we obtain are quite different. $V_4$ is abelian and order 4, whereas $D_8 \times C_2$ is non-abelian and order 16. 
In fact, $V_4$ is a subgroup of $D_8 \times C_2$ -- and actually also of $D_8$ alone. Both matrix factorizations prove the same orbifold equivalence, but the second version has the advantage that the resulting equations are much easier to handle.
\end{rem}

\quad It would certainly be interesting to further explore the connection between Galois groups and matrix factorizations proving orbifold equivalence between potentials described by singularities. That is the aim of the second part of this paper, \cite{NRC}. Some ideas we would like to explore are the following.

We intend to investigate whether it is possible to predict from the outset whether a given expression of a singularity will lead to a Galois group which is easy to handle (e.g. abelian). In the particular example we have dealt with in this paper, in Version 1 the potential for $K_{14}$ had a cross term, whereas in Version 2 (the easier one), the potential for $K_{14}$ had only pure power monomials. But as we have seen in Table \ref{Excepts1}, not all the candidates for orbifold equivalence which at the same time are strangely dual have an associated potential which only has pure power monomials. In our case indeed a simpler shape of the potentials led to a simpler Galois group, but further analysis of other cases may give us some hints about how the Galois groups vary for each expression of the potentials.



While proving orbifold equivalence, in both \cite{CRCR} and this paper we observe the repeated appearance of $C_2$ in the resulting Galois groups. We would like to investigate whether this is a coincidence or there is some intrinsic relationship with the structure of matrix factorizations. 


\quad Altogether, we look for(ward to) a better understanding of the orbifold equivalence, and we hope to provide further insights very soon.

\newpage
\begin{appendix}
\section{Counting invertible potentials -- by Ana Ros Camacho and Federico Zerbini}\label{appendix}

Besides the Arnold classification, one may ask the following question: given a polynomial ring with $n$ variables over the complex numbers, how many kinds of potentials can we have and what do they look like?

A partial answer is provided by the Kreuzer-Skarke theorem \cite{KS,HK}. In these papers they provide a graphical algorithm to generate potentials that we recall here.

Fix a regular set of weights. We call a \textit{configuration} the set of polynomials in $\mathbb{C} \left[x_1,\ldots,x_n \right]$ with this regular set of weights. A classification of potentials is encoded in certain graphs representing configurations. Every variable is represented by a dot, and a term of the form $x_i^a x_j$ is represented by an arrow from $x_i$ to $x_j$ (``\textit{$x_i$ points at $x_j$''}).

\begin{defn}
We call a variable $x_i$ a \textit{root} if the polynomial $W$ contains a term $x_i^a$. A monomial $x_j^a x_k$ is called a \textit{pointer} at $x_k$. The number $a$ is called the \textit{exponent} of $x_i$ or $x_j$, respectively. We recursively define a link between two expressions, which may themselves be variables or links, as a monomial depending only on the variables occurring in these expressions. A link may further be linear in additional variables, which don't count as variables of the link. In this case we say that the link points at $x_k$, extending the definition of a pointer. It is possible that a specific monomial could have more than one interpretation as a link or a pointer. Given a potential $W$, any graph whose lines allow the above interpretation in terms of monomials in $W$ is a graphic representation of $W$.
\end{defn}

The following result is taken verbatim from \cite{KS}.
\begin{thm}\footnote{This theorem has been reformulated in a slightly more general setting in \cite{HK}, but we keep here the original formulation from \cite{KS} as the graphical language proves intuitive and useful for explanations.}
\label{thm:KS}
For a configuration a necessary and sufficient condition for a polynomial to be a potential is that it has a member which can be represented by a graph where:
\begin{enumerate}
\item Each variable is either a root or points at another variable.
\item For any pair of variables and/or links pointing at the same variable $x_i$ there is a link joining the two pointers and not pointing at $x_i$ or any of the targets of the sublinks which are joined\footnote{We will draw these links as dotted arrows to distinguish them from those coming from the first condition of the theorem.}.
\end{enumerate}

\end{thm}

Let us explain how this theorem works presenting a couple of examples for a small number of variables:
\begin{itemize}
\item \cite{Ar,AGV} For $n=2$, we find three graphs:
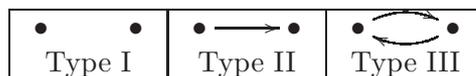
\begin{center}
\begin{tabular}{|c|c|c|}
\hline
$\xymatrix{ \bullet & \bullet}$ & $\xymatrix{ \bullet  \ar[r] & \bullet}$ & $\xymatrix{ \bullet \ar@/^/[r] & \bullet \ar@/^/[l]}$ \\ 
Type I & Type II & Type III \\ \hline
\end{tabular}
\end{center}

\item \cite{Ar,AGV} For $n=3$, we find seven graphs as specified in Table \ref{3vars}.
\end{itemize}
\begin{table}
\begin{center}
\begin{tabular}{|c|c|c|c|}
\hline
$\xymatrix{ & \bullet & \\ \bullet && \bullet}$ & $\xymatrix{ & \bullet \ar[dr] & \\ \bullet && \bullet}$ & $\xymatrix{ & \bullet \ar@/^/[dr]& \\ \bullet && \bullet \ar@/^/[ul]}$ & $\xymatrix{ & \bullet \ar[dr] & \\ \bullet \ar[ur] && \bullet}$ \\ 
Type I & Type II & Type III & Type IV \\ \hline
$\xymatrix{ & \bullet \ar[dr] & \\ \bullet \ar[ur] && \bullet \ar[ll]}$ & $\xymatrix{ & \bullet \ar[dr] \ar@{.}[dl] & \\ \bullet \ar[rr] && \bullet}$ & $\xymatrix{ & \bullet \ar@/^/[dr] \ar@{.}[dl] & \\ \bullet \ar[rr] && \bullet \ar@/^/[ul]}$ & \\
Type V & Type VI & Type VII &  \\ \hline 
\end{tabular}
\caption{Types of potentials for $n=3$.}
\label{3vars}
\end{center}
\end{table}

\begin{rem}
Notice that for $n=3$ the second condition of Theorem \ref{thm:KS} is only relevant for Types VI and VII. Actually, one can reformulate this second condition for Types VI and VII as follows \cite{AGV}.  Every potential of Type VI contains a monomial in $\lbrace x^a,y^b x,z^c x \rbrace$,  and those of Type VII contain a monomial in $\lbrace x^a y,y^b x,z^c x \rbrace$ (up to suitable changes of variables). The exponents of these potentials must satisfy the following conditions:
\begin{itemize}
\item Type VI: the least common multiple of $b$ and $c$ must be divisible by $a-1$.
\item Type VII: $\left( b-1 \right) c$ must be divisible by the product of $a-1$ and the greatest common divisor of $b$ and $c$.
\end{itemize}
\end{rem}

The potentials generated via this theorem can be divided into two classes:

\begin{defn}
\begin{itemize}
\item Let $W$ be a potential. We say $W$ is \textit{invertible} when the following conditions are satisfied:
\begin{itemize}
\item The number of variables $n$ coincides with the number of monomials in $W$, $$W \left( x_1,\ldots,x_n \right)= \sum_{i=1}^n a_i \prod_{j=1}^n x_j^{E_{ij}}$$ for some coefficients $a_i \in \mathbbm{C}^*$ and $E_{ij} \in \mathbb{Z}_{\geq 0}$.
\item The matrix $E :=\left( E_{ij} \right)$ is invertible over $\mathbb{Q}$.
\item \cite{BH} The \textit{Berglund-H{\"u}bsch transpose} of $W$, written $W^T$ and defined by $$W^T \left( x_1,\ldots,x_n \right)=\sum_{i=1}^n a_i \prod_{j=1}^n x_j^{E_{ji}},$$ is also a potential.
\end{itemize}
\item If a potential is not invertible, we call it a \textit{beserker}. 
\end{itemize}
\end{defn}

As an example, notice that a potential in two variables is always invertible. In three variables, it is invertible if it is of type I--V, and it is a beserker if it is of type VI or VII.

\begin{rem}
\begin{itemize}
\item The Berglund-H{\"u}bsch transposition is closely related to mirror symmetry and the Landau-Ginzburg/Calabi-Yau correspondence, see for example \cite{Chi} or \cite{ET}. 
\item For the potentials associated to the singularities $Q_{10}$ and $K_{14}$, notice that in Version \ref{version1} we present they are Berglund-H{\"u}bsch transposes of each other, while this is not the case in Version \ref{version2}.
 Actually, whenever we take the Berglund-H{\"u}bsch transpose of a potential (from the first column of Table \ref{Excepts1})  described by an exceptional unimodular singularity, we either obtain the same potential or the corresponding strange dual.
\item In addition, notice that the Berglund-H{\"u}bsch transposition preserves the central charge for invertible potentials \cite{ARC}, which suggests that Berglund-H{\"u}bsch may be a source of orbifold equivalences (see Remark \ref{charges}).
\item For invertible potentials, the Berglund-H{\"u}bsch transposition corresponds graphically to reversing the directions of the arrows.
\end{itemize}
\end{rem}

\begin{rem}\label{invtypes}
Invertible potentials can only be of three types (or combinations of them) \cite{KS}:
\begin{itemize}
\item \textit{Fermat}: $x_1^{a_1}+x_2^{a_2}+\ldots+x_n^{a_n}$
\item \textit{Chain}: $x_1^{a_1} x_2+x_2^{a_2} x_3+\ldots+x_{n-1}^{a_{n-1}}+x_n^{a_n}$
\item \textit{Loop}: $x_1^{a_1} x_2+x_2^{a_2} x_3+\ldots+x_{n-1}^{a_{n-1}}+x_n^{a_n} x_1$
\end{itemize}

Translating this in terms of dots and arrows, the Fermat part of the potentials is represented by isolated dots (see Type I in Table \ref{3vars}), the chain part by the union of all chains, i.e. the sequences of arrows leading from one dot to another distinct dot (see Type IV), and the loop part by the union of all loops, i.e. the sequences of arrows leading from one dot to itself (see Type V).
This means that the invertible potentials are in one-to-one correspondence with mappings of $n$ points to themselves that never involve two points mapping to a third distinct point. 
\end{rem}
A question that may arise at this point is, using this description in terms of mappings of points, how many invertible potentials do we get for a given number of variables?

We denote by $P(n)$ the number of partitions of $n$, and we denote by $P_1(n)$ the number of partitions of $n$ where we exclude parts with cardinality $1$. For example $P(3)=3$, because we have the partitions $\{3\}$, $\{2,1\}$ and $\{1,1,1\}$, but $P_1(3)=1$, since only $\{3\}$ is allowed.\\
It is easy to see that the two sequence are related by $P_1(n)=P(n)-P(n-1)$.
\begin{prop}
The number of invertible potentials (or, for brevity, invertibles) is given by 
\begin{equation}\label{inv}
\mathrm{Inv}(n)=1+2\sum_{k=2}^nP_1(k)+\sum_{k=4}^n\sum_{i=2}^{k-2}P_1(i)P_1(k-i).
\end{equation}
\end{prop}

\begin{proof}
First, notice that the number of invertibles given by Equation \ref{inv} matches our computation by hand:
\begin{equation}
\begin{split}
\mathrm{Inv} \left( 2 \right) &=3 \\
\mathrm{Inv} \left( 3 \right) &=5 \\
\mathrm{Inv} \left( 4 \right) &=10 \\
\mathrm{Inv} \left( 5 \right) &=16 \\
\mathrm{Inv} \left( 6 \right) &=29 \\
\mathrm{Inv} \left( 7 \right) &=45 \\
\mathrm{Inv} \left( 8 \right) &=75 \\
\mathrm{Inv} \left( 9 \right) &=115 \\
\end{split}
\nonumber
\end{equation}
The proof relies on the fact that one can easily compute $\mathrm{Inv}(n)$ if $\mathrm{Inv}(n-1)$ is given. This is  because every invertible with $n$ dots that has at least one isolated dot can be thought as an invertible with $n-1$ dots plus the mentioned isolated one. This means that counting the invertibles without isolated dots is the same as computing $\mathrm{Inv}(n)-\mathrm{Inv}(n-1)$.\\
One can think of an invertible without isolated dots as divided into 2 blocks: one constituted by chains and one constituted by loops. Note that the number of dots in any chain or loop is at least 2, so one gets the following intuitive formula:
\[
\mathrm{Inv}(n)-\mathrm{Inv}(n-1)=2P_1(n)+\sum_{i=2}^{n-2}P_1(i)P_1(n-i),
\]
where $2P_1(n)$ counts the invertibles constituted either only by chains or only by loops (this is why there is a factor $2$!), and the sum counts the invertibles with a mix of chains and loops.
Now the proof of Equation \ref{inv} is trivial, because we already now that it works for the first values of $n$, so we just need to check that Equation \ref{inv} also gives the difference predicted above, which is straightforward.
\end{proof}

\begin{rem}\label{rem}
(Courtesy of G. Sanna) This formula can be rewritten as 
\begin{equation}\label{gian}
\mathrm{Inv}(n)=\sum_{k=0}^n P(n-k)[P(k)-P(k-1)],
\end{equation}
with $P(0):=1$, $P(-1):=0$. One can easily prove that the two formulas give the same result by induction, rewriting $P(n-k)$ as $P_1(n-k)+P(n-1-k)$ in Equation \ref{gian} and using the fact that $P_1(1)=0$ and that $P_1(k)=P(k)-P(k-1)$.
\end{rem}

Thanks to Remark \ref{rem}\footnote{Actually, the generating function was found originally using Equation \ref{inv} and without Equation \ref{gian}, but the proof was less elegant.}, one can immediately see what is the generating function for the numbers $\mathrm{Inv}(n)$:
\begin{cor}
Setting $\mathrm{Inv}(0):=1$, we have
\begin{equation}\label{gen}
\sum_{n\geq 0}\mathrm{Inv}(n)q^n=(1-q)\prod_{m\geq 1}(1-q^m)^{-2}
\end{equation}
\end{cor}
\begin{proof}
Expanding every term in the product on the right as a power series in $q$ shows that
\begin{equation*}
\sum_{n\geq 0}P(n)q^n=\prod_{m\geq 1}(1-q^m)^{-1}.
\end{equation*}
So the left hand side of Equation \ref{gen} can be rewritten as
\[
(1-q)\Big(\sum_{n\geq 0}P(n)q^n\Big)^2.
\]
The result now follows from the observation that
\[
\Big(\sum_{n\geq 0}P(n)q^n\Big)^2=\sum_{n\geq 0}\Big(\sum_{k=0}^nP(n-k)P(k)\Big)
\]
and
\[
q\Big(\sum_{n\geq 0}P(n)q^n\Big)^2=\sum_{n\geq 0}\Big(\sum_{k=0}^nP(n-k)P(k-1)\Big).
\]
\end{proof}
Note that this is the same generating function as the one generating the sequence A000990 in the Encyclopedia of Integer Sequences (\cite{Slo}), which counts the number of plane partitions of n with at most two rows.

\end{appendix}

\newpage

\newcommand\arxiv[2]      {\href{http://arXiv.org/abs/#1}{#2}}
\newcommand\doi[2]        {\href{http://dx.doi.org/#1}{#2}}
\newcommand\httpurl[2]    {\href{http://#1}{#2}}

\end{document}